\documentclass{amsart}
\usepackage{graphicx}
\usepackage{amssymb}

\newtheorem{thm}{Theorem}[section]

\newtheorem{lem}[thm]{Lemma}
\newtheorem{prop}[thm]{Proposition}
\theoremstyle{definition}

\theoremstyle{remark}
\newtheorem{rem}[thm]{Remark}
\theoremstyle{conjecture}
\newtheorem{conj}[thm]{Conjecture}

\begin{document}

\title[Small degeneracy set]{Minimizers of convex functionals with small degeneracy set}
\author{Connor Mooney}
\address{Department of Mathematics, UC Irvine}
\email{\tt mooneycr@math.uci.edu}

\begin{abstract}
We study the question whether Lipschitz minimizers of $\int F(\nabla u)\,dx$ in $\mathbb{R}^n$ are $C^1$ when $F$ is strictly convex. Building on work of De Silva-Savin, we confirm the $C^1$ regularity
when $D^2F$ is positive and bounded away from finitely many points that lie in a $2$-plane. We then construct a counterexample in $\mathbb{R}^4$, where $F$ is strictly convex but $D^2F$ degenerates 
on the intersection of a Simons cone with $S^3$. Finally we highlight a connection between the case $n = 3$ and a result of Alexandrov in classical differential geometry, and we make a conjecture about this case.
\end{abstract}
\maketitle

\section{Introduction}
In this paper we study the regularity of Lipschitz minimizers of
\begin{equation}\label{Functional}
E(u) = \int_{B_1} F(\nabla u)\,dx
\end{equation}
in $\mathbb{R}^n$, where $F: \mathbb{R}^n \rightarrow \mathbb{R}$ is convex. By Lipschitz minimizer we mean a function $u \in W^{1,\,\infty}(B_1)$ that satisfies
$E(u + \varphi) \geq E(u)$ for all $\varphi \in C^1_0(B_1)$. It is straightforward to show that Lipschitz minimizers solve the Euler-Lagrange equation
\begin{equation}\label{EL}
\text{div}(\nabla F(\nabla u)) = 0
\end{equation}
in the weak sense. Conversely, any Lipschitz weak solution of (\ref{EL}) is a minimizer of $E$ by the convexity of $F$.

In the extreme case that the graph of $F$ contains a line segment, minimizers are no better than Lipschitz by simple examples. 
In the other extreme that $F$ is smooth and uniformly convex, De Giorgi and Nash proved that Lipschitz minimizers are smooth and solve the Euler-Lagrange equation $F_{ij}(\nabla u)u_{ij} = 0$
classically (\cite{DG}, \cite{Na}). It remains largely open what happens in the intermediate case where $F$ is strictly convex, but the eigenvalues of $D^2F$ go to $0$ or $\infty$ on some set $D_F$. Such functionals arise naturally in the study of anisotropic surface tensions \cite{DMMN}, traffic flow \cite{CF}, and statistical mechanics (\cite{CKP}, \cite{KOS}).

In \cite{DS} the authors raise the natural question: 
\begin{equation}\label{Question}
\text{\it Are Lipschitz minimizers in $C^1$ when $F$ is strictly convex?}
\end{equation}
They give evidence that the answer may be ``yes," at least in two dimensions. In particular, they show that if $n = 2$ and $D_F$ consists of finitely many points,
then Lipschitz minimizers of $E$ are $C^1$. In this paper we study this question in higher dimensions. We first confirm
the $C^1$ regularity of Lipschitz minimizers when $D_F$ is a finite set in some $2$-plane. In particular, this covers the case that $D_F$ consists of three points. 
We then show the answer to Question (\ref{Question}) is ``no" in general, by constructing a singular Lipschitz minimizer in $\mathbb{R}^4$. In our example, $F$ is in fact uniformly convex and $C^1$, but one eigenvalue of $D^2F$ goes to $\infty$ on the intersection of a Simons cone with $S^3$. This leaves open the possibility that Lipschitz minimizers are $C^1$ in dimension $n \geq 3$ in the interesting case that $D_F$ consists of finitely many points. To address this problem we connect it to a result of Alexandrov in the classical differential geometry of convex surfaces, and we propose a possible counterexample
in $\mathbb{R}^3$ where $D_F$ consists of four non-coplanar points.

\begin{rem}
Guided by the observation that the Legendre transform $F^*$ of $F$ solves $\text{div}(\nabla F(\nabla F^*)) = \text{div}(x) = n$, one could (more ambitiously) ask whether the minimizers
are as regular as $F^*$. This is known in some special cases, e.g. for the $p$-Laplace case $F(x) = |x|^p$ when $p > 2$ and $n = 2$ (see \cite{ATO},\, \cite{IM}).
\end{rem}

\begin{rem}
The case that $D_F$ consists of a single point (e.g. $p$-Laplace) is well-studied (see \cite{E}, \cite{Uh}, \cite{Ur}). The case that $D_F$ is ``large" is also understood: in \cite{CF} the authors show 
that if $D_F$ is convex and $F = 0$ on $D_F$, then for $x \in B_1$ the gradients $\nabla u(B_r(x))$ localize as $r \rightarrow 0$ either to a point outside $D_F$ or to $D_F$.
\end{rem}

\begin{rem}
One can show the existence of Lipschitz minimizers with additional hypotheses on the behavior of $F$ at infinity. For example, if $F$ has quadratic growth,
then for $g \in H^1(B_1)$ the direct method gives the existence of a minimizer $u \in H^1(B_1)$ with $u - g \in H^1_0(B_1)$. If $g$ is smooth enough
($C^{1,\,1}$ suffices) then $u$ is Lipschitz by the comparison principle. Alternatively, if $F$ is uniformly convex with bounded second derivatives at infinity, then $u$ is locally Lipschitz. For a proof of this result, see \cite{Ma2}. 
The local Lipschitz regularity of minimizers is in fact true under assumptions that allow for growth of $D^2F$ at infinity (the so called $(p,\,q)$ growth conditions); see \cite{Ma1}, \cite{Ma2}, and the references therein.
\end{rem}

The paper is organized as follows. In Section \ref{Statements} we give precise statements of our results, and we discuss a connection between the problem in dimension $n = 3$ and a result
of Alexandrov. In Section \ref{C1Proof} we prove the $C^1$ regularity result. In Section \ref{CounterexProof} we construct the counterexample. Finally, in the Appendix
we record some technical results that we used to construct the counterexample.

\section{Statements of Results}\label{Statements}
Let $F : \mathbb{R}^n \rightarrow \mathbb{R}$ be a $C^1$ convex function, and let $D_F \subset \mathbb{R}^n$ be a compact set such that
$$F \in C^2(\mathbb{R}^n \backslash D_F), \quad D_F = \mathbb{R}^n \backslash \left(\cup_{k > 1} \{k^{-1}I < D^2F < kI\}\right).$$
Here and below, dependence on $F$ means dependence on the sets 
$$\mathcal{O}_k := \{k^{-1}I < D^2F < kI\} \subset \mathbb{R}^n \backslash D_F$$ 
(in particular, the geometry of $D_F$), and on the moduli continuity of $D^2F$ in compact sets that exhaust $\mathbb{R}^n \backslash D_F$. Our first theorem is:
\begin{thm}\label{C1}
Let $u$ be a Lipschitz solution of (\ref{EL}). If $D_F$ is finite and is contained in a two-dimensional affine subspace of $\mathbb{R}^n$, then $u \in C^1(B_1)$,
and the modulus of continuity of $\nabla u$ in $B_{1/2}$ depends only on on $n,\,F,$ and $\|\nabla u\|_{L^{\infty}(B_1)}$.
\end{thm}

\begin{rem}
We conjecture that the assumption in Theorem \ref{C1} is optimal. That is, that there exists a singular minimizer in $\mathbb{R}^3$ where $D_F$ consists of four non-coplanar points (see the discussion in Section
\ref{Hedgehogs}).
\end{rem}

The starting point of Theorem \ref{C1} is the well-known fact that convex functions of $\nabla u$ are sub-solutions to the linearized Euler-Lagrange equation. Using this 
fact we show that $\nabla u(B_r)$ localizes as $r \rightarrow 0$ either to a point outside $D_F$ (in which case we are done), or to the convex hull of $D_F$.
This was observed in \cite{CF} in the case that $D_F$ is a convex set and $F = 0$ on $D_F$, motivated by models of traffic congestion.
The key observation in \cite{DS} is that in two dimensions, certain slightly non-convex functions of $\nabla u$ are also sub-solutions to the linearized equation.
If the convex hull of $D_F$ is two-dimensional, we can use higher-dimensional versions of these functions to further localize the gradients to a point.

\vspace{2mm}

To state our second result we let $x = (x_1,\,x_2) \in \mathbb{R}^{2n}$ with $x_i \in \mathbb{R}^n$. We define
\begin{equation}\label{Minimizer}
w(x) := \frac{1}{\sqrt{2}} \frac{|x_2|^2 - |x_1|^2}{|x|}.
\end{equation}
Then $w$ is a nontrivial one-homogeneous function on $\mathbb{R}^{2n}$ that is analytic outside of the origin. We show:

\begin{thm}\label{Counterex}
When $n \geq 2$, $w$ is a minimizer of a functional of the form (\ref{Functional}) with $F$ uniformly convex and $C^1$, and $D_F = \{|x_1|^2 = |x_2|^2\} \cap \sqrt{2}\,S^{2n-1}$.
\end{thm}

Our approach to Theorem \ref{Counterex} is based on the observation that when $n \geq 2$, the gradient image $\Sigma_w := \nabla w(\mathbb{R}^{2n} \backslash \{0\}) = \nabla w(S^{2n-1})$ 
is a saddle-shaped hypersurface that is smooth away from a ``cusp" singularity on $\{|x_1|^2 = |x_2|^2\} \cap \sqrt{2}\, S^{2n-1}$. This reflects that $D^2w$ has positive and negative eigenvalues, and thus
solves some elliptic equation. We then build the integrand $F$ near $\Sigma_w$ so that the Euler-Lagrange equation (\ref{EL}) is satisfied, and finally we make a global convex extension.
In previous work with Savin we took a similar approach to construct singular minimizers of functionals with large degeneracy set in $\mathbb{R}^3$, where
$D_F$ consists of two disconnected convex sets with nonempty interior \cite{MS}.


\subsection{The Case $n = 3$ and Hyperbolic Hedgehogs}\label{Hedgehogs}

To conclude the section we highlight a connection between our approach to Theorem \ref{Counterex} and classical differential geometry.

\vspace{2mm} 

Natural candidates for singular minimizers are one-homogeneous functions with Hessians that have indefinite sign. Indeed, such functions are invariant under the rescalings that preserve (\ref{EL}), 
and they solve some elliptic PDE. It is useful to identify a one-homogeneous function $u$ with its gradient image, a (possibly singular)
hypersurface $\Sigma_u$. The function $u$ is the support function of $\Sigma_u$, and the eigenvalues of $D^2u$ on $S^{n-1}$ are the principal radii of $\Sigma_u$.
The set $\Sigma_u$ is the parallel set a distance $A$ in the direction of the {\it inward} unit normal from the convex body $\Sigma_{u + A|x|}$, where $A$ is chosen large
enough that $D^2u + AI > 0$ on $S^{n-1}$. Such parallel surfaces to a convex body are known in the literature as ``hedgehogs" (see e.g. \cite{MM2}).

\vspace{2mm}

In dimension $n = 3$, a natural candidate for a singular minimizer thus corresponds to a hedgehog that is saddle-shaped away from its singularities, i.e. a parallel set a distance $A$ in the inward direction from a convex surface with
principal radii $r_1,\,r_2 > 0$ that satisfy $(r_1 - A)(r_2-A) \leq 0$. Alexandrov originally conjectured that the only such convex surfaces in $\mathbb{R}^3$ are spheres. He proved his conjecture for 
analytic convex surfaces (\cite{A1}, \cite{A2}). Thus, we cannot construct with our method a singular minimizer in $\mathbb{R}^3$ that is analytic outside of the origin (compare to Theorem \ref{Counterex}). 
For $C^{2}$ surfaces, Alexandrov's conjecture remained open for a long time (with at least one incorrect proof). It was resolved in $2001$ by a beautiful counterexample of Martinez-Maure (\cite{MM1}).
Martinez-Maure's hedgehog is built by gluing together four self-intersecting ``cross caps" with figure-eight cross sections that shrink to cusps (see Figure \ref{Hedgehog}). 
Motivated by this discussion and Theorem \ref{C1} we conjecture:

\begin{figure}
\begin{center}
    \includegraphics[scale=0.4, trim={0 7cm 30mm 0}, clip]{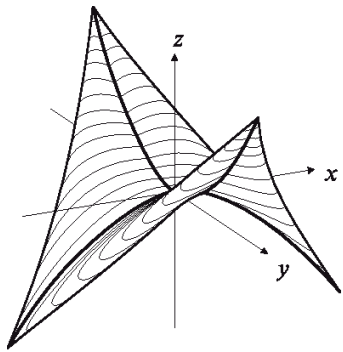}
    \caption{The hyperbolic hedgehog $\Sigma_h$ from \cite{MM1}.}
    \label{Hedgehog}
  \end{center}
\end{figure}

\begin{conj}\label{FourPointConjecture}
The support function $h$ of the hedgehog from \cite{MM1} is a one-homogeneous singular minimizer of a functional of the type (\ref{Functional}), where $D_F$ consists exactly of the four
cusps of $\Sigma_h$.
\end{conj}
\noindent This would show that the geometric conditions on $D_F$ in Theorem $\ref{C1}$ are optimal. 

\begin{rem}
The surface $\Sigma_h$ can be written as the union of two graphs, which makes writing the Euler-Lagrange equation on 
$\Sigma_h$ relatively simple. Using this observation we can show that it is possible to construct $F$ locally (in particular, in a small neighborhood of a cusp), with some tedious 
calculation. It seems challenging to construct $F$ globally, but so far we do not see a fundamental obstruction.
\end{rem}

\begin{rem}
The regularity of $h$ in the example from (\cite{MM1}) is $C^2$. Smooth counterexamples to Alexandrov's conjecture, with (a version of) $\Sigma_h$ as a special case, were later constructed by Panina \cite{P}.
\end{rem}

\section{Proof of Theorem \ref{C1}}\label{C1Proof}
Choose $M_0$ large so that $D_F \subset B_{M_0}$, and let $M = \max\{M_0,\, \|u\|_{L^{\infty}(B_1)}\}$.
By a standard approximation argument, to prove Theorem \ref{C1} it suffices to assume $u,\,F \in C^{\infty}$ and show that the modulus of continuity of $\nabla u$ in $B_{1/2}$ depends only on $M$, the sets $\{\mathcal{O}_k\}$, and the moduli of continuity of $D^2F$ in the sets $\{B_M \cap \mathcal{O}_k\}$ (see e.g. \cite{CF}).

\subsection{Preliminaries}
We record some important preliminary results. Our argument is based on applying the following estimate of De Giorgi (the ``weak Harnack inequality") to various functions of $\nabla u$:
\begin{prop}\label{WeakHarnack}
Assume that $v \geq 0$ is in $H^1(B_1)$ and solves $\partial_i(a_{ij}(x)v_j) \geq 0$, with $a_{ij}$ bounded measurable and $\lambda I \leq (a_{ij}) \leq \lambda^{-1} I$ for some $\lambda > 0$.
Then for all $\mu > 0$, there exists $\nu(\mu,\,n,\,\lambda) > 0$ such that if 
$$\frac{|\{v > 0\} \cap B_1|}{|B_1|} \leq 1-\mu$$
then
$$\sup_{B_{1/2}} v \leq (1-\nu)\sup_{B_1}v.$$
\end{prop}
\noindent To prove Proposition \ref{WeakHarnack} apply the weak Harnack inequality for supersolutions (Theorem $8.18$ in \cite{GT}) to $\sup_{B_1}v - v$.

\vspace{2mm}

We now discuss the types of functions of $\nabla u$ that Proposition \ref{WeakHarnack} applies to. We denote the linearized Euler-Lagrange operator by $L_F$. That is,
$$L_F(v) := \text{div}(D^2F(\nabla u)\nabla v) = \partial_i(F_{ij}(\nabla u)v_j).$$ 
The key observation is that if $\eta$ is slightly concave in only one direction, then $\eta(\nabla u)$ is a subsolution of $L_F$ where $\nabla u$ avoids $D_F$.
For $\Omega \subset \mathbb{R}^n$ let $\mathcal{N}_{\delta}(\Omega)$ denote the $\delta$-neighborhood of $\Omega$. We have:
\begin{lem}\label{NonConvex}
Assume $\eta$ is a smooth function in a neighborhood of $\nabla u(B_1)$.
For any $\rho \in (0,\,1)$, there exists $\lambda(\rho,\,F,\,M,\,n) > 0$ such that if $\nabla u(B_1) \cap \{\eta > 0\} \subset B_M \backslash \mathcal{N}_{\rho}(D_F)$, and
in $\nabla u(B_1) \cap \{\eta > 0\}$ the eigenvalues $\gamma_1 \leq \gamma_2 \leq ... \leq \gamma_n$ of $D^2\eta$ satisfy $\gamma_2 > 0$ and $\gamma_1 \geq -\lambda \gamma_2$, then
$$L_F(\eta_+(\nabla u)) \geq 0.$$
\end{lem}
\noindent Here $\eta_+ := \max\{\eta,\,0\}$.
\begin{proof}
Using that $L_F(u_k) = 0$ we compute
$$L_F(\eta(\nabla u)) = \text{div}(D^2F(\nabla u) \nabla (\eta(\nabla u))) = F_{ij}u_{jk}\eta_{kl}u_{li}.$$
At a fixed point $x_0 \in \{\eta(\nabla u) > 0\}$ choose coordinates so that $\eta_{kl}(\nabla u(x_0)) = \gamma_k\delta_{kl}$. Summing over $l$ we obtain
$$L_F(\eta(\nabla u))(x_0) = \sum_{k = 1}^n \gamma_k \sum_{i,\,j = 1}^n F_{ij}u_{ki}u_{kj}.$$
For some $m$ large depending on $\rho,\,F,\,M$ we have $\nabla u(B_1) \cap \{\eta > 0\} \subset \mathcal{O}_m.$ 
Since $\gamma_1 \geq -\lambda \gamma_2$ in $\{\eta(\nabla u) > 0\}$ we conclude that
$$\gamma_2^{-1}L_F(\eta(\nabla u))(x_0) \geq m^{-1} \sum_{k = 2}^{n} |\nabla u_k(x_0)|^2 - m\,\lambda |\nabla u_1(x_0)|^2.$$
If $L_F(\eta(\nabla u))(x_0) < 0$ then the above inequality gives
$$(\lambda^{-1}m^{-2}-1) \sum_{(i,\,j) \neq (1,\,1)} u_{ij}^2(x_0) < u_{11}^2(x_0).$$
On the other hand, the equation $F_{11}u_{11} = -F_{ij}u_{ij}|_{(i,\,j) \neq (1,\,1)}$ gives
$$u_{11}^2 \leq C(n,\,m)\sum_{(i,\,j) \neq (1,\,1)} u_{ij}^2$$
in $\{\eta(\nabla u) > 0\}$. The previous two inequalities contradict each other for $\lambda(n,\,m)$ small.
\end{proof}

\vspace{2mm}


In order to apply Proposition \ref{WeakHarnack} and Lemma \ref{NonConvex}, we need $\nabla u$ to be close to $D_F$ in sets of positive measure.
The alternative is that $u$ nearly solves a non-degenerate equation. To handle this situation we will also use a ``flatness implies regularity" result for $u$:

\begin{prop}\label{SmallPert}
Assume that $a_{ij}$ are smooth elliptic coefficients on $\mathbb{R}^n$ that satisfy $\lambda I \leq (a_{ij}) \leq \lambda^{-1} I$ in $B_{\rho}(p)$ for some fixed $\lambda,\, \rho > 0,\, p \in \mathbb{R}^n$.
There exists $\epsilon > 0$ depending on $\rho,\,n,\,\lambda$ and the modulus of continuity of $a_{ij}$ in $B_{\rho}(p)$ such that if
$v \in C^{\infty}(B_1)$ solves $a_{ij}(\nabla v)v_{ij} = 0$ and
$$\|v - l_p\|_{L^{\infty}(B_1)} \leq \epsilon$$
for some linear function $l_p$ with $\nabla l_p = p$, then
$$\nabla v(B_{1/2}) \subset B_{\rho}(p).$$
\end{prop}
\noindent Heuristically, $w := \epsilon^{-1}(v-l_p)$ solves $a_{ij}(p + \epsilon \nabla w)w_{ij}  = 0$ which is nearly a constant-coefficient equation for $\epsilon$ small.
The idea of Proposition \ref{SmallPert} is due to Savin \cite{S}, who treated equations with degeneracy in the Hessian of $v$. For a proof of the proposition as stated (with degeneracy in the gradient of $v$) see \cite{CF}. 

An easy consequence of Proposition \ref{SmallPert} is:
\begin{lem}\label{Linearization}
For any $\rho \in (0,\,1)$ there exist $\epsilon_1,\,\mu_1(\rho,\,F,\,M,\,n) > 0$ such that if 
$$\frac{|\{\nabla u \in B_{\epsilon_1}(p)\} \cap B_r|}{|B_r|} \geq 1-\mu_1$$
for some $p \in B_{2M} \backslash \mathcal{N}_{2\rho}(D_F)$ then
$$\nabla u(B_{r/2}) \subset B_{\rho}(p).$$
\end{lem}
\begin{proof}
After taking $u \rightarrow r^{-1}u(r x)$ we may assume that $r= 1$. Since $u$ solves $F_{ij}(\nabla u)u_{ij} = 0$,
by Proposition \ref{SmallPert} there exists $\epsilon_0 > 0$ depending on $\rho,\,F,\,M,\,n$ such that if
$$\|u - l_p\|_{L^{\infty}(B_1)} \leq \epsilon_0,$$
for some linear function $l_p$ with $\nabla l_p = p$, then $\nabla u(B_{1/2}) \subset B_{\rho}(p)$. The above inequality holds by standard embeddings if we take e.g. $\epsilon_1 < c(n)\epsilon_0$ and
take $\mu_1$ small depending on $M,\,n,\,\epsilon_0$.
\end{proof}

Our approach to Theorem \ref{C1} is to first show that as $r \rightarrow 0$, the sets $\nabla u(B_r)$ localize to the convex hull of $D_F$, and
then to show that if this set is two-dimensional, they localize to a point. We treat these two results separately in the following sub-sections, and then combine them.

\subsection{Localization to the Convex Hull}
Let $K_F$ denote the convex hull of $D_F$. In this subsection we show:
\begin{prop}\label{ConvexLocalization}
For any $\rho > 0$, there exists $s(\rho,\,F,\,M,\,n) > 0$ such that either $\nabla u(B_{s}) \subset B_{\rho}(p)$ for some $p \in B_{2M} \backslash \mathcal{N}_{2\rho}(D_F)$, or $\nabla u(B_{s}) \subset \mathcal{N}_{4\rho}(K_F)$.
\end{prop}

\noindent In this subsection we call a constant universal if it depends only on $\rho,\,M,\,F,\,n$.
Let $\beta$ be a smooth uniformly convex function on $\mathbb{R}^n$ such that 
$$\mathcal{N}_{3\rho}(K_F) \subset \{\beta \leq 0\} \subset \mathcal{N}_{4\rho}(K_F),$$ 
with $B_{2M} \subset \{\beta < \tilde{M}\}$ for some universal $\tilde{M}$. Let $\epsilon_1,\,\mu_1 > 0$ be the universal constants from Lemma \ref{Linearization},
corresponding to $\rho$.

\begin{lem}\label{Chopping}
There exists $\delta > 0$ universal such that if $\sup_{B_r}\beta(\nabla u) > 0$ and
$$\frac{|\{\nabla u \in B_{\epsilon_1}(p)\} \cap B_r|}{|B_r|} < 1-\mu_1$$
for all $p \in B_{2M} \backslash \mathcal{N}_{2\rho}(D_F)$, then 
$$\sup_{B_{r/2}}\beta(\nabla u) \leq \sup_{B_r}\beta(\nabla u) - \delta.$$
\end{lem}
\begin{proof}
After taking $u \rightarrow r^{-1}u(r x)$ we may assume that $r = 1$.
Let $0 < t := \sup_{B_1}\beta(\nabla u) \leq \tilde{M}$. For any unit vector $e$ let $s_e := \sup_{\{\beta < t\}} (p \cdot e)$. There is some universal $\delta_0 > 0$ (independent of $e$) such that for some $c_e \leq s_e - \delta_0$ we have 
$$\text{diam}(\{p \cdot e \geq c_e\} \cap \{\beta < t\}) < \epsilon_1$$ 
and that $D^2F$ has universal ellipticity constant in $\{p \cdot c_e > 0\}$.
By the hypotheses we may apply Proposition \ref{WeakHarnack} to $v := (u_e - c_e)_+$ with $\mu = \mu_1$ to conclude that $u_e \leq (1-\nu_1)s_e + \nu_1 c_e \leq s_e - \nu_1\delta_0$ in $B_{1/2}$, with $\nu_1 > 0$ universal. (Here we use that the coefficients of $L_F$ have universal ellipticity constant in $\{v > 0\}$; we can replace
the coefficients by e.g. $\delta_{ij}$ in $\{v = 0\}$ without changing the equation for $v$).
Since $\cap_{e \in S^{n-1}} \{p \cdot e \leq s_e - \nu_1\delta_0\} \subset \{\beta < t - \delta\}$ for some universal $\delta > 0$ the proof is complete.
\end{proof}

\begin{proof}[{\bf Proof of Proposition \ref{ConvexLocalization}}]
Apply the following algorithm for $k \geq 0$: if one of the hypotheses of Lemma \ref{Chopping} is not satisfied in $B_{2^{-k}}$, then stop. We either have $\nabla u(B_{2^{-k}})
\subset \{\beta \leq 0\} \subset \mathcal{N}_{4\rho}(K_F)$, or we can apply Lemma \ref{Linearization} to conclude $\nabla u(B_{2^{-k-1}}) \subset B_{\rho}(p)$ for some $p \in B_{2M} \backslash \mathcal{N}_{2\rho}(D_F)$.
Otherwise, we apply Lemma \ref{Chopping}. The algorithm terminates after at most $k_0$ steps with $\tilde{M} - k_0 \delta \leq 0$.
\end{proof}

\subsection{Localization Beyond the Convex Hull}
In this subsection we show that if $\nabla u(B_1)$ is sufficiently close to a two-dimensional affine subspace, then 
as $r \rightarrow 0$ the gradients $\nabla u(B_r)$ localize to a connected component of $D_F$. 

Let $(p,\,q) \in \mathbb{R}^n$ with $p \in \mathbb{R}^2$ and $q \in \mathbb{R}^{n-2}$, and assume that $D_F \subset \{q = 0\}$.
Let $8\rho_0$ be the smallest distance between a pair of points in $D_F$. In this subsection we call constants depending on $\rho_0,\,F,\,M,\,n$ universal.

\begin{prop}\label{Connected}
There exist $\sigma_0,\,r_0 > 0$ universal such that if $\nabla u(B_1) \subset \{|q| < \sigma_0\}$ then either $\nabla u(B_{r_0}) \subset
B_{\rho_0}(p)$ for some $p \in B_{2M} \backslash \mathcal{N}_{2\rho_0}(D_F)$ or $\nabla u(B_{r_0}) \subset \mathcal{N}_{3\rho_0}(D_F)$.
\end{prop}

\noindent In particular, since $\nabla u$ is (qualitatively) continuous the set $\nabla u(B_{r_0})$ is connected, so it is contained in a ball with at most one point of $D_F$, and is a distance at least $\rho_0$ from the remaining points in $D_F$. 

The idea is to localize the gradients using the level sets of non-convex functions of $\nabla u$.
Let $\epsilon_1,\,\mu_1 > 0$ be the (universal) constants from Lemma \ref{Linearization} with $\rho = \rho_0$. We assume by taking $\epsilon_1$ smaller
if necessary that $\epsilon_1 \leq \rho_0$. Let $\lambda_1$ be the constant from Lemma \ref{NonConvex} with $\rho = \rho_0$.
Finally, let $\nu_1$ be the constant from Proposition \ref{WeakHarnack} corresponding to $\mu = \mu_1$ and the ellipticity constants of $D^2F$ in $B_{2M} \backslash \mathcal{N}_{\rho_0}(D_F)$. The following lemma says that when the gradient image is sufficiently close to $\{q = 0\}$, we can ``chop" at its projection to $\{q = 0\}$ with circles (see Figure \ref{ChoppingPic}):

\begin{lem}\label{Chopping2}
Let $(p_0,\,0) \in B_{2M} \backslash \mathcal{N}_{2\rho_0}(D_F)$. There exist $\sigma_0,\,\delta_0 > 0$ universal such that if
$\nabla u(B_r) \subset \{|q| < \sigma_0\} \cap \{|p-p_0| \geq \epsilon_1/4\}$ and 
$$\frac{|\{\nabla u \in B_{\epsilon_1}(p_0,\,0)\} \cap B_r|}{|B_r|} < 1 - \mu_1,$$
then
$$\nabla u(B_{r/2}) \subset \{|p - p_0| \geq \epsilon_1/4 + \delta_0\}.$$
\end{lem}
\begin{proof}
We may assume that $r = 1$ after a Lipschitz rescaling. Define 
$$\eta_A(p,\,q) := e^{A^2|q|^2/2 - A|p|}, \quad \eta_{A,\,p_0} := \eta_A(p - p_0,\,q) - e^{-A\epsilon_1/2}.$$
In an appropriate system of coordinates we have in $\{|q| < A^{-3}\}$ that
$$(A^2\eta_A)^{-1}D^2\eta_A = \text{diag}(-(A|p|)^{-1},\,1,\,...,\,1) + O(A^{-2}),$$
and that $\{\eta_A > e^{-A\epsilon_1/2}\} \subset \{|p| < \epsilon_1/2 + A^{-5}\}$. Then by our first hypothesis, for $A$ large universal and $\sigma_0 < A^{-3}$
we have $\nabla u(B_1) \cap \{\eta_{A,\,p_0} > 0\} \subset B_M \cap B_{\epsilon_1}(p_0,\,0) \subset B_M \backslash \mathcal{N}_{\rho_0}(D_F)$, and
that the eigenvalues $\gamma_1 \leq ... \leq \gamma_n$ of $D^2\eta_{A,\,p_0}$ satisfy $\gamma_2 > 0$ and $\gamma_1 > -\lambda_1\gamma_2$
in $\nabla u(B_1) \cap \{\eta_{A,\,p_0} > 0\}$. We conclude using Lemma \ref{NonConvex} that the function $v_{p_0} := (\eta_{A,\,p_0})_+(\nabla u)$
satisfies $L_F(v_{p_0}) \geq 0$. By our second hypothesis we can apply Proposition \ref{WeakHarnack} to $v_{p_0}$. 
In the extreme case that $\sigma_0 = 0$, Proposition \ref{WeakHarnack} gives
$$\nabla u(B_{1/2}) \subset \{\eta_{A,\,p_0} < (1-\nu_1)(e^{-A\epsilon_1/4} - e^{-A\epsilon_1/2})\} \subset \{|p - p_0| \geq \epsilon_1/4 + 2\delta_0\}$$
for some $\delta_0 > 0$ universal. By continuity we have the same inclusion with $2\delta_0$ replaced by $\delta_0$ for sufficiently small
$\sigma_0 < A^{-3}$, completing the proof.
\end{proof}

\begin{figure}
 \centering
    \includegraphics[scale=0.4]{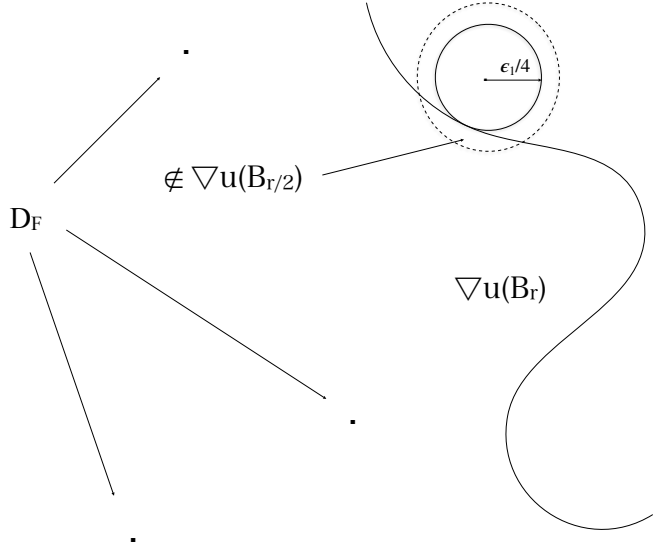}
 \caption{If $\nabla u(B_r)$ is ``nearly $2D$" we can chop its projection to $\{q = 0\}$ with circles. The picture illustrates the case $n = 2$.}
 \label{ChoppingPic}
\end{figure}

We can now prove Proposition \ref{Connected}.
\begin{proof}[{\bf Proof of Proposition \ref{Connected}:}]
Take $\sigma_0$ as in Lemma \ref{Chopping2}, and apply the following algorithm for $k \geq 0$: If 
$$\frac{|\{\nabla u \in B_{\epsilon_1}(p_0,\,0)\} \cap B_{2^{-k}}|}{|B_{2^{-k}}|} \geq 1-\mu_1$$
for some $(p_0,\,0) \in B_{2M} \backslash \mathcal{N}_{2\rho_0}(D_F)$, then stop. We have by Lemma \ref{Linearization} that  $\nabla u(B_{2^{-k-1}}) \subset B_{\rho_0}(p_0,\,0)$. If not, apply Lemma \ref{Chopping2} to conclude that
$$\nabla u(B_{2^{-k-1}}) \subset \{|p-p_0| \geq \epsilon_1/4 + \delta_0\}$$
for all $(p_0,\,0) \in B_{2M} \backslash \mathcal{N}_{2\rho_0}(D_F)$ such that $\nabla u(B_{2^{-k}}) \subset \{|p - p_0| \geq \epsilon_1/4\}$.
If at this point we can conclude that $\nabla u(B_{2^{-k-1}}) \subset \mathcal{N}_{3\rho_0}(D_F)$, then stop.

To show that this algorithm terminates after a universal number of steps, we use a simple covering argument. Let $\mathcal{S}_k$ be the projections of the sets $\nabla u(B_{2^{-k}})$ to $\{q = 0\}$. Take a finite number of lines $\{l_i\}$ in $\{q = 0\}$ that avoid $\mathcal{N}_{2\rho_0}(D_F)$, whose $\epsilon_1/4$ neighborhoods cover $B_{2M} \backslash \mathcal{N}_{2\rho_0}(D_F) \cap \{q = 0\}$. For each $l_i$ take a universal number $J+1$ of two-dimensional balls $\{B_{ij} := B_{\epsilon_1/4}(p_{ij})\}_{j = 1}^{J+1}$ in $\{q = 0\}$ that cover $B_{2M} \cap l_i \cap \{q = 0\}$, with centers $p_{ij} \in B_{2M} \cap l_i \cap \{q = 0\}$ on $l_i$ such that $|p_{i,\,j+1} - p_{ij}| \leq \delta_0$ for $j =1,\,...,\,J$. Since $\mathcal{S}_0 \subset B_M$ we can arrange that $B_{i1} \cap \mathcal{S}_0 = \emptyset$ for all $i$. By induction, if the algorithm doesn't terminate after $k$ steps, 
then $\mathcal{S}_{k}$ has empty intersection with the convex hull of $B_{i1}$ and $B_{i,\,k+1}$, for each $i$. In particular, after $J$ steps we have that $\mathcal{S}_{J} \subset\mathcal{N}_{2\rho_0}(D_F)$, and the proof is complete up to replacing $\sigma_0$ with $\min\{\sigma_0,\, \rho_0\}$.
\end{proof}

\subsection{Proof of Theorem \ref{C1}}
We are now in position to prove Theorem \ref{C1}. We call constants depending on $F,\,M,\,n$ universal.
\begin{proof}
For any $\epsilon > 0$ we will show that there is some $\delta(\epsilon,\,F,\,M,\,n) > 0$ such that $\nabla u(B_{\delta})$ is contained in a ball of radius $\epsilon$.

Take $\sigma_0$ to be the constant from Proposition \ref{Connected}. Applying Proposition \ref{ConvexLocalization} with $\rho = \sigma_0/4$ we obtain
$s_0 > 0$ universal such that either $\nabla u(B_{s_0}) \subset B_{\sigma_0/4}(p)$
for some $p \in B_{2M} \backslash \mathcal{N}_{\sigma_0/2}(D_F)$, or $\nabla u(B_{s_0}) \subset \mathcal{N}_{\sigma_0}(K_F)$.
In the latter case, apply Proposition \ref{Connected} to $s_0^{-1}u(s_0x)$ to conclude for some $r_0 > 0$ universal that either $\nabla u(B_{r_0s_0}) \subset B_{\rho_0}(p)$ for some $p \in B_{2M} \backslash \mathcal{N}_{2\rho_0}(D_F)$ or $\nabla u(B_{r_0s_0}) \subset \mathcal{N}_{3\rho_0}(D_F)$.

In all cases, $\nabla u(B_{r_0s_0})$ is contained in a ball $\mathcal{B}$ that has at most one point of $D_F$ and is a positive universal distance from the remaining
points of $D_F$. Thus, after restricting our attention to $\tilde{u} = (r_0s_0)^{-1}u(r_0s_0 x)$ we may assume that
$D_F$ contains at most one point (indeed, we can modify $F$ outside of $\mathcal{B}$ without changing that $\tilde{u}$ is a minimizer). Applying Proposition \ref{ConvexLocalization} to $\tilde{u}$ with $\rho = \epsilon / 4$ completes the proof.
\end{proof}
\section{Proof of Theorem \ref{Counterex}}\label{CounterexProof}

In this section we construct the examples from Theorem \ref{Counterex}.
Here and below we let $k \geq 1$, and $q = (q_1,\,q_2),\, y = (y_1,\,y_2) \in \mathbb{R}^{2k+2}$ with $q_i,\,y_i \in \mathbb{R}^{k+1},\, i = 1,\,2.$
We will reduce the problem to making a certain one-dimensional construction using the symmetries of $w$.


\subsection{Reduction to Two Dimensions}
We first reduce Theorem \ref{Counterex} to a problem in two dimensions. Let $v$ be the one-homogeneous function on $\mathbb{R}^2$ given by
$$v(x_1,\,x_2) := \frac{1}{\sqrt{2}} \frac{x_2^2 - x_1^2}{|x|}.$$
We claim it suffices to construct a $C^1$, uniformly convex function $G(p_1,\,p_2)$ on $\mathbb{R}^2$ that is smooth away from $D_{G} = \sqrt{2}\,S^1 \cap \{p_1^2 = p_2^2\}$, 
such that $G$ is invariant under reflection over the axes and over the lines $\{p_1 = \pm p_2\}$ (that is, $G(p_1,\,p_2) = G(-p_1,\,p_2) = G(p_2,\,p_1)$) and furthermore
\begin{equation}\label{ELReduction1}
\text{tr}(D^2G(\nabla v)\,D^2v) + k\,\nabla G(\nabla v) \cdot \left(\frac{1}{x_1},\,\frac{1}{x_2}\right) = 0
\end{equation}
for $x$ in the positive quadrant. Indeed, if we manage to do this, note that by the symmetries of $G$ and $v$, each term on the left is smooth away from $\{x_1^2 = x_2^2\}$, where $\nabla v$ maps to $D_G$.
If we then take $F(q) = G(|q_1|,\,|q_2|)$ we obtain a $C^1$, uniformly convex function on $\mathbb{R}^{2k+2}$ that is smooth away from $D_F = \sqrt{2}\,S^{2k+1} \cap \{|q_1|^2 = |q_2|^2\}$. Using that $w(y) = v(|y_1|,\,|y_2|)$ we compute
\begin{align*}
\text{tr}(D^2F(\nabla w)\,D^2w)(y) &=  \text{tr}(D^2G(\nabla v)\,D^2v)(|y_1|,\,|y_2|) \\
&+ k\,\nabla G(\nabla v(|y_1|,\,|y_2|) )\cdot \left(\frac{1}{|y_1|},\,\frac{1}{|y_2|}\right) \\
&= 0
\end{align*}
classically away from the cone $\{|y_1|^2 = |y_2|^2\}$. Here we used that $v_1 < 0$ and $v_2 > 0$ in the positive quadrant. It is not hard to show that the equation $\text{div}(\nabla F(\nabla w)) = 0$ 
holds in the weak sense in $B_1$ by integrating away from a thin cone containing $\{|y_1|^2 = |y_2|^2\}$ and a small ball around the origin, using the $C^1$ regularity of $F$ and the one-homogeneity of $w$,
and taking a limit.


\subsection{Reduction to One Dimension}
We now use that $\Sigma_v := \nabla v(S^{1})$ is one-dimensional and an extension lemma to reduce our problem to one dimension.
The set $\nabla v(S^1 \cap \{x_2 \geq |x_1|\})$ can be written as a graph $\Gamma_1 := \{(p_1,\, \varphi(p_1))\}$ with $p_1 \in [-1,\,1]$, where $\varphi \in C^{\infty}(-1,\,1) \cap C^1([-1,\,1])$ is 
even, uniformly convex, and separates from the lines $p_2 = \pm p_1$ like $\text{dist.}^{3/2}$ at the endpoints. See the Appendix for a justification of these properties, as well
as an expansion of $\varphi$ near the endpoints. The set $\Sigma_v$ consists
of four rotations of $\Gamma_1$ by $\pi/2$ (see Figure \ref{Fig2}). Let $S := \sqrt{2}\,S^1 \cap \{p_1^2 = p_2^2\}$, and let $\Sigma_0 := \Sigma_v \backslash S,\, \Gamma_0 := \Gamma_1 \backslash S$. 
We will use the following important extension lemma:

\begin{figure}
 \centering
    \includegraphics[scale=0.35]{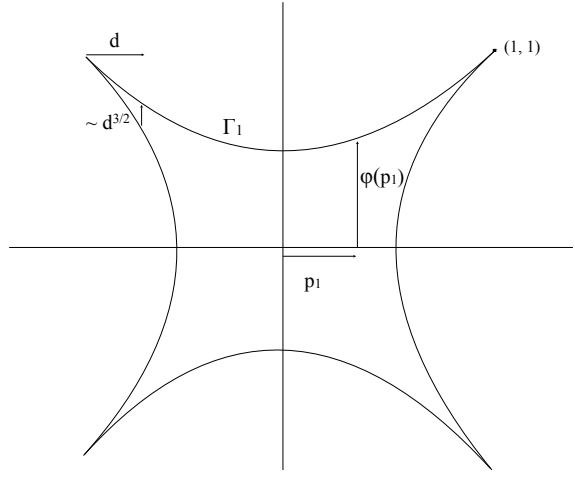}
 \caption{The set $\Sigma_v$ consists of four congruent curves separating from the lines $p_2 = \pm p_1$ like $\text{dist}^{3/2}.$}
 \label{Fig2}
\end{figure}

\begin{lem}\label{ExtensionLemma}
Assume that $g : \Sigma_v \rightarrow \mathbb{R}$ and ${\bf v} : \Sigma_v \rightarrow \mathbb{R}^2$ are smooth on $\Sigma_0$ and continuous on $\Sigma_v$, and satisfy the condition
\begin{equation}\label{tangsep1}
g(\tilde{p}) - g(p) - {\bf v}(p)\cdot (\tilde{p} - p) \geq \gamma |\tilde{p} - p|^2
\end{equation}
for some $\gamma > 0$ and all $\tilde{p},\, p \in \Sigma_v$. Then there exists a $C^1$, uniformly convex function $G$ on $\mathbb{R}^2$ with $G \in C^{\infty}(\mathbb{R}^2 \backslash S)$, such that 
$G = g$ and $\nabla G = {\bf v}$ on $\Sigma_v$.
\end{lem}

\noindent We delay the proof of Lemma \ref{ExtensionLemma} to the Appendix, and proceed with the reduction.
We claim that it suffices to construct even functions $f,\,h \in C^{\infty}(-1,\,1) \cap C^1([-1,\,1])$ such that 
\begin{equation}\label{ELReduction2}
h = \frac{1}{1+2k}\left(\frac{f''}{\varphi''} + k\frac{f'}{\varphi'}\right)
\end{equation}
on $(-1,\,1)$, and furthermore the pair $(g,\,{\bf v})$ defined on $\Sigma_v$ by
\begin{equation}\label{TangentData}
g(s,\,\varphi(s)) := f(s), \quad {\bf v}(s,\,\varphi(s)) := (f'(s) - h(s)\varphi'(s),\, h(s))
\end{equation}
on $\Gamma_1$ and extended by reflection over the lines $p_2 = \pm p_1$, satisfies the conditions of Lemma \ref{ExtensionLemma} for some $\gamma > 0$.
Indeed, if this is accomplished, then the extension $G$ from Lemma \ref{ExtensionLemma} satisfies the Euler-Lagrange equation (\ref{ELReduction1}).
To see this, let $\nu$ be the upward unit normal to $\Gamma_0$. By the one-homogeneity of $v$ we have
$$x = \nu(\nabla v(x)) = \frac{(-\varphi',\,1)}{\sqrt{1 + \varphi'^2}}$$ 
for $x \in S^1 \cap \{x_2 > |x_1|\}$. Differentiating in $x$ gives 
$$D^2v(x) = \frac{1}{\kappa}\tau \otimes \tau = -\frac{(1+ \varphi'^2)^{3/2}}{\varphi''}\tau \otimes \tau$$ 
where $\tau$ is the unit tangent vector to $\Gamma_0$ at $\nabla v(x)$ (also the unit tangent to $S^1$ at $x$) and $\kappa$ is the (signed) curvature of $\Gamma_0$. 
Finally, by (\ref{TangentData}) we have
\begin{equation}\label{Restrictions}
G(s,\,\varphi(s)) = f(s), \quad G_2(s,\,\varphi(s)) = h(s)
\end{equation}
on $\Gamma_0$. Differentiating the first relation in (\ref{Restrictions}) twice we obtain
$$\nabla G = (f' - h\varphi',\,h), \quad \tau^T \cdot D^2G \cdot \tau = \frac{f''-h\varphi''}{1 + \varphi'^2}$$
on $\Gamma_0$. Putting these together gives the equivalence of (\ref{ELReduction1}) and (\ref{ELReduction2}).
Finally, since $(G,\, \nabla G)$ has the desired symmetries on $\Sigma_v$, we can arrange that $G$ has the desired symmetries globally by taking
the average of its reflections.

\begin{rem}
One can compute the Euler-Lagrange equation (\ref{ELReduction2}) directly in $\mathbb{R}^{2k+2}$ using the geometry of $\Sigma_w := \nabla w(S^{2k+1})$, without much trouble. 
Using the one-homogeneity of $w$ we see as above that the equation reduces to $\text{tr}((II)^{-1}(q)\,D_T^2F(q)) = 0$ on $\Sigma_w$, where $II$ is the second fundamental form of $\Sigma_w$ and $T$ is the tangent hyperplane to $\Sigma_w$. Since $\Sigma_w$ is obtained by taking rotations of $\Gamma_1$, it is tangent on one side to second order to a sphere of radius $\varphi\sqrt{1 + \varphi'^2}$ and on the other side to second order to a sphere of radius $\frac{s}{\varphi'}\sqrt{1 + \varphi'^2}$. It thus has one principal curvature $-\varphi''\,(1+ \varphi'^2)^{-3/2}$, $k$ principal curvatures $\frac{1}{\varphi\,\sqrt{1 + \varphi'^2}}$ from rotating around the $p_1$ axis, and $k$ principal curvatures $-\frac{\varphi'}{s\sqrt{1 + \varphi'^2}}$ from rotating around the $p_2$ axis. The eigenvalues of $D^2F$ corresponding to these directions are $\frac{f'' - h\varphi''}{1 + \varphi'^2},\, \frac{h}{\varphi}$ and $\frac{f'-h\varphi'}{s}$, where the latter two come from rotations. Putting these together in the original Euler-Lagrange equation $\text{tr}((II)^{-1}(q)\,D_T^2F(q)) = 0$ we recover (\ref{ELReduction2}).
\end{rem}

\subsection{The One-Dimensional Construction} 

We now construct $f$ and $h$. The Euler-Lagrange equation (\ref{ELReduction2}) determines $h$ through our choice of $f$, so there is only one function to construct.
It is convenient to do this by taking
\begin{equation}\label{EtaDef}
f''(s) := \eta(s)\varphi''(s), \, f'(0) = 0
\end{equation}
for some positive $\eta \in C^{\infty}(-1,\,1) \cap C([-1,\,1])$. Fix $0 < \delta << 1$. We choose $\eta$ satisfying the following conditions:
\begin{enumerate}
\item[(i)] $\eta$ is even, concave, and $\eta(1) = 1/2$,
\item[(ii)] $\eta \geq \min\left\{1,\, \frac{1}{2}(1 + \delta^{-1/2}(1-s)^{1/2})\right\}$,
\item[(iii)] $\eta \equiv 1 + \mu$ on $[0,\,1-\delta]$ for some $\mu > 0$,
\item[(iv)] $\int_0^1 \eta(s)\varphi''(s)\,ds = 1$.
\end{enumerate}
Since $\int_0^1 \varphi''(s)\,ds = 1$, it is clear that $\mu \rightarrow 0$ as we take $\delta \rightarrow 0$. We will show below that for any such choice of $\eta$ with $\delta$ sufficiently small, the pair $(g,\,{\bf v})$
defined by (\ref{TangentData}) satisfies the hypotheses of Lemma \ref{ExtensionLemma}.

\vspace{2mm}

{\bf Continuity Condition:} The condition that ${\bf v}$ is continuous on $\Sigma_v$ and invariant under reflection over the diagonal is that
\begin{equation}\label{C1Condition}
(f' - h\varphi',\, h)(1) \text{ is parallel to } (1,\,1).
\end{equation}
This follows from (\ref{ELReduction2}), using that $\eta(1) = 1/2$ and $f'(1) = \varphi'(1) = 1$.

\vspace{2mm}

{\bf Convexity Condition Along Top Graph}. 
We check the convexity condition (\ref{tangsep1}) on $\Gamma_1 \cap \{p_1 \geq 0\}$. Take $\tilde{p} = (y,\,\varphi(y))$ and $p = (x,\,\varphi(x))$ for some $x,\,y \in [0,\,1]$. The quantity of interest is
\begin{align*}
g(y,\,\varphi(y)) &- g(x,\,\varphi(x)) - {\bf v}(x,\,\varphi(x)) \cdot (y-x,\,\varphi(y)-\varphi(x))  \\
&= [f(y) - f(x) - f'(x)(y-x)] - h(x) [\varphi(y) - \varphi(x) - \varphi'(x) (y-x)] \\
&= \int_x^y f''(x)(y-s)\,ds - h(x) \int_x^y \varphi''(s)(y-s)\,ds \\
&= \left(\int_x^y \varphi''(s)(y-s)\,ds \right)H(x,\,y),
\end{align*}
where
$$H(x,\,y) = \int_x^y (\eta-1/2)(s)d\mu_y(s) - \frac{1}{1+2k}(\eta - 1/2)(x) - \frac{k}{1+2k}\frac{f' - \varphi'}{\varphi'}(x),$$
and $d\mu_y$ is the probability density
$$d\mu_y := \frac{\varphi''(s)(y-s)\,ds}{\int_x^y \varphi''(t)(y-t)\,dt}$$
on the interval from $x$ to $y$. We claim that
\begin{equation}\label{HBound}
H(x,\,y) \geq c_0\,\max\{(1-x)^{1/2},\, (1-y)^{1/2}\}
\end{equation}
for some $c_0 > 0$ independent of $x,\,y \in [0,\,1]$. The desired inequality (\ref{tangsep1}) then follows because 
$$\frac{1}{(y-x)^2} \int_x^y \varphi''(s)(y-s)\,ds \geq c\, \min\{(1-x)^{-1/2},\, (1-y)^{-1/2}\}$$
for some $c > 0$, using that $\varphi''$ is positive, increasing on $[0,\,1)$, and has the expansion $\varphi''(s) = \sqrt{\frac{2}{3}}(1-s)^{-1/2} + O(1)$ near $s = 1$ (see Appendix). 
To show (\ref{HBound}) we check several cases. Below, $c_1,\,c_2$ always denote universal positive constants which may change from line to line.

In the case that $x,\,y \leq 1-\delta$ it is obvious that $H(x,\,y) \geq c_0 > 0$ since $\eta-1/2 > 1/2$ is constant on $[0,\,1-\delta]$ and $f'-\varphi' = \mu \varphi'$ with $\mu$ small.

The next simplest case is that $y < x$ and $x > 1-\delta$. Using that $\eta$ is decreasing, that $f' \leq 1$, and the expansion of $\varphi'(s)$ near $s = 1$ (see Appendix) we have
$$H(x,\,y) \geq \frac{2}{3} \int_x^y (\eta - 1/2)(s)d\mu_y(s) - c_1\,(1-x)^{1/2}.$$
It is elementary to check that the mass of $d\mu_y$ in the left half of the interval $[y,\,x]$ is at least $c_2 > 0$ independent of $y$ using the properties of $\varphi$.
Since $\eta$ is decreasing and concave we conclude that
$$H(x,\,y) \geq \frac{1}{3}c_2\,(\eta - 1/2)(y) - c_1\,(1-x)^{1/2} \geq c_0(1-y)^{1/2}$$
using definition of $\eta$ and that $\delta$ is small.

The most delicate case is that $x < y$ and $y > 1-\delta$. By the expansion of $\varphi(s)$ near $s = 1$ we may choose $\delta$ so small
that $d\mu_y$ is decreasing on $[1-2\delta,\,y]$, independent of $y > 1-\delta$ (see Appendix). We first claim that
$$\int_x^y (\eta-1/2)(s)d\mu_y(s) \geq \frac{1}{2}(\eta - 1/2)(x).$$
If $x < 1-2\delta$ then since most of the weight of $d\mu_y$ on $[x,\,y]$ is to the left of $1-\delta$ the inequality is obvious. 
When $x \geq 1-2\delta$, since both the weight $d\mu_y$ and $\eta$ are decreasing on $[x,\,y]$, the average of $\eta - 1/2$ decreases if we redistribute the weight $d\mu_y$ evenly, and
the concavity of $\eta$ gives the result. We conclude that
$$H(x,\,y) \geq \frac{1}{6}(\eta - 1/2)(x) - \frac{f'-\varphi'}{2\varphi'}(x).$$
If $x \leq 1-\delta$ we have $f' - \varphi' = \mu \varphi'$ and $\eta - 1/2 > 1/2$ so $H(x,\,y) \geq c_0 > 0$. If $x > 1-\delta$ we argue as in the case $y < x$ that 
$$H(x,\,y) \geq \frac{1}{12} \delta^{-1/2}(1-x)^{1/2} - c_1(1-x)^{1/2} > (1-x)^{1/2}$$
for $\delta$ small, and inequality (\ref{HBound}) follows.

\vspace{2mm}

{\bf Global Convexity Condition.}
Finally, for the convexity condition (\ref{tangsep1}) to hold on all of $\Sigma_v$, it suffices by reflection symmetry to show that
\begin{equation}\label{Reflection}
{\bf v}(s,\,\varphi(s)) \cdot (1,\,0) \geq c_0 s, \quad {\bf v}(s,\,\varphi(s)) \cdot (-1,\,1) \geq c_0(\varphi(s) - s)
\end{equation}
for some fixed $c_0 > 0$ and all $s \in [0,\,1]$.

For $\delta$ small it is straightforward to show that $h < 3/4$, so $f' - h\varphi' \geq c_0\varphi'$, and the first inequality in (\ref{Reflection}) follows.

For the second we compute
\begin{align*}
(1+2k){\bf v} \cdot (-1,\,1) &= (1+2k)[(1 + \varphi')h - f'] \\
&= (1+\varphi')\eta + [k(1 + 1/\varphi') - (1+2k)]f' \\
&= (1+\varphi')\eta + [k(1/\varphi' - 1) - 1]f' \\
&\geq (1+\varphi')\eta - f'. 
\end{align*}
When $0 \leq s \leq 1-\delta$ this quantity is larger than $1$ by the definition of $\eta$. For $s \geq 1-\delta$ we use the expansion
$\varphi' = 1 - c_1\,(1-s)^{1/2} + O(1-s)$ (see Appendix) and that $f' \leq 1$ to get
\begin{align*}
(1+\varphi')\eta - f' &\geq (2-c_1(1-s)^{1/2})\eta - 1 + O(1-s) \\
&= (2\eta - 1) - c_1\eta(1-s)^{1/2} + O(1-s) \\
&\geq (\delta^{-1/2} - 2c_1)(1-s)^{1/2} + O(1-s) \\
&\geq \frac{1}{2}\delta^{-1/2}(1-s)^{1/2}
\end{align*}
for $\delta$ small. We conclude that
$${\bf v} \cdot (-1,\,1) \geq c_0 \min\{1,\, \delta^{-1/2}(1-s)^{1/2}\}.$$
Since $\varphi(s) - s \leq c_2(1-s)^{3/2}$ for some $c_2 > 0$, this confirms (\ref{Reflection}).


\subsection{Proof of Theorem \ref{Counterex}}
\begin{proof}
Choose $\eta$ as in the previous subsection. Then the pair $(g,\,{\bf v})$ determined by $\eta$ through the relations (\ref{EtaDef}), (\ref{ELReduction2}), and
(\ref{TangentData}) satisfies the hypotheses of Lemma \ref{ExtensionLemma}. We showed above that the extension $G$ then satisfies 
the Euler-Lagrange equation (\ref{ELReduction1}) and can be chosen symmetric over the axes and over the lines $p_1 = \pm p_2$, 
and that the result follows by taking $F(q) = G(|q_1|,\,|q_2|)$.
\end{proof}

\begin{rem}\label{FReg}
If we choose $\eta - 1/2$ to be a multiple of $(1-s)^{1/2}$ near $s = 1$, then a straightforward computation shows that the $2k + 1$ second derivatives of $F$ in directions tangent
to $\Sigma_w$ are bounded, and the only second derivative of $F$ that tends to $\infty$ is the one normal to $\Sigma_w$.
The regularity of $F$ near $D_F$ is in fact $C^{1,\,\frac{1}{3}}$, by the computation that confirms the second inequality in (\ref{Reflection}).
\end{rem}


\section{Appendix}\label{Appendix}
In the Appendix we record some properties of $\varphi$, and we prove the extension result Lemma \ref{ExtensionLemma}.

\subsection{Properties of $\varphi$}\label{Expansions}
We recall from \cite{MS} that if we parametrize $\Gamma_1$ by the angle $\theta \in [\pi/4,\,3\pi/4]$ of its upward unit normal $\nu$, then its curvature is given by
$\kappa = \frac{\sqrt{2}}{3}\sec(2\theta)$. It follows easily that $\varphi$ is smooth, even, and uniformly convex on $(-1,\,1)$, and $\varphi''$ is increasing on $[0,\,1)$. We recall also the expansion
$$\varphi\left(-1 + \frac{3}{2}\theta^2 + \theta^3 + O(\theta^4)\right) = 1-\frac{3}{2}\theta^2 + \theta^3 + O(\theta^4)$$
from \cite{MS}. By differentiating implicitly and using that $\varphi$ is even we obtain near $s = 1$ the expansions
$$\varphi'(s) = 1 - 2\sqrt{\frac{2}{3}}(1-s)^{1/2} + O(1-s)$$
$$\varphi''(s) = \sqrt{\frac{2}{3}}(1-s)^{-1/2} + O(1)$$
$$\varphi'''(s) = \frac{1}{2}\sqrt{\frac{2}{3}}(1-s)^{-3/2} + O((1-s)^{-1}).$$
In particular, for $0 < s < y$ the derivative of the weight $\varphi''(s)(y-s)$ is bounded above by $\varphi'''(s)(1-s) - \varphi''(s) = -\frac{1}{2}\sqrt{\frac{2}{3}}(1-s)^{-1/2} + O(1) < 0$ for $s$ close to $1$.


\subsection{Proof of Extension Lemma}

We now prove Lemma \ref{ExtensionLemma}. Our strategy is to first construct $G$ in a set containing a neighborhood of every point on $\Sigma_0$. We then apply a global $C^1$ extension result to
this local extension. To complete the construction we use a mollification and gluing procedure.

\subsubsection{Local Extension}

\begin{lem}\label{LocalExtension}
There exists an open set $\mathcal{O}$ containing a neighborhood of each point on $\Sigma_0$ and a function $G_0 \in C^{\infty}(\mathcal{O})$ such that 
$G_0 = g$ and $\nabla G_0 = {\bf v}$ on $\Sigma_0$, and furthermore
\begin{equation}\label{tangsep2}
G_0(\tilde{p}) - G_0(p) - \nabla G_0(p) \cdot (\tilde{p} - p) \geq \frac{\gamma}{2}|\tilde{p} - p|^2
\end{equation}
for all $p,\,\tilde{p} \in \mathcal{O}$.
\end{lem}
\begin{proof}
The squared distance function $d_{\Sigma_v}^2$ from $\Sigma_v$ is smooth in a neighborhood of each point on $\Sigma_0$, as is the projection $\pi_{\Sigma_v}$ to $\Sigma_v$. Let $\tau$
be a unit tangent vector field to $\Sigma_v$ in a neighborhood of a point on $\Sigma_0$, and $\nu$ a unit normal vector field.
Then $D^2(d_{\Sigma_v}^2/2)$ projects in the normal direction $\nu$ on $\Sigma_0$. See e.g. \cite{AS} for proofs of these properties.

Let $A > 0$ be a smooth function on $\Sigma_0$ to be chosen, and define
$$G_0(x) := g(\pi_{\Sigma_v}(x)) + {\bf v}(\pi_{\Sigma_v}(x)) \cdot (x - \pi_{\Sigma_v}(x)) + A(\pi_{\Sigma_v}(x))d_{\Sigma_v}^2.$$
It is elementary to check using (\ref{tangsep1}) that $\nabla_{\Sigma_0}g$ is the tangential component of ${\bf v}$ on $\Sigma_0$, and
as a consequence that $G_0 = g,\, \nabla G_0 = {\bf v}$ on $\Sigma_0$. Furthermore, it follows from (\ref{tangsep1}) that $(G_0)_{\tau\tau} \geq 2\gamma$ on $\Sigma_0$.
Since $D^2(A(\pi_{\Sigma_v}(x))d_{\Sigma_v}^2)$ is $2A$ times the matrix that projects in the direction $\nu$ on $\Sigma_0$, we have that $(G_0)_{\nu\nu} = 2A$ on $\Sigma_0$,
and that $(G_0)_{\tau\nu}$ on $\Sigma_0$ depends only on $g,\,{\bf v}$, and the geometry of $\Sigma_0$ (in particular, not on $A$). By choosing $A(p)$ sufficiently large depending
on $\gamma$ and these quantities (and perhaps going to $\infty$ as $p \rightarrow S$), we have that 
\begin{equation}\label{tangsep3}
D^2G_0 \geq \frac{3}{2}\gamma I
\end{equation}
on $\Sigma_0$. In particular, (\ref{tangsep2}) holds in a small neighborhood of each point on $\Sigma_0$.

Now, for $\delta,\,\sigma > 0$ let $S_{\delta}$ be the closed $\delta$-neighborhood of $S$, and let $\Sigma_{\sigma}$ be the open $\sigma$-neighborhood of $\Sigma_v$. By
(\ref{tangsep1}), (\ref{tangsep3}), and continuity, for
each $\delta > 0$ there exists $\sigma(\delta) > 0$ small and an open set $O_{\delta}$ containing $\Sigma_{\sigma(\delta)} \backslash S_{\delta}$ and a neighborhood of each
point in $\Sigma_0$, such that (\ref{tangsep2}) holds for all $p,\,\tilde{p} \in O_{\delta}$ such that at least one of $p,\,\tilde{p}$ is in $\Sigma_{\sigma(\delta)} \backslash S_{\delta}$.
For $\delta = 1/k$ we may choose $\sigma(1/k)$ and $O_{1/k}$ such that $\Sigma_{\sigma(1/k)} \backslash S_{1/k} \subset O_{1/k} \subset O_{1/(k-1)}$ for $k > 1$.
Taking
$$\mathcal{O} = \cup_{k > 1} (\Sigma_{\sigma(1/k)} \backslash S_{1/k})$$
completes the proof.
\end{proof}

\subsubsection{Global $C^1$ Extension}

\begin{lem}\label{C1Extension}
There exists a $C^1$ convex function $G_1$ on $\mathbb{R}^2$ such that $G_1 = G_0$ on an open set $\mathcal{U} \subset \mathcal{O}$ that contains a neighborhood of each
point on $\Sigma_0$, such that $D^2G_1 \geq \frac{\gamma}{2}I$ on $\mathbb{R}^2$.
\end{lem}
\begin{proof}
Let $K$ be a compact set containing $\Sigma_v$ and a neighborhood of each point on $\Sigma_0$, such that $K \backslash S \subset \mathcal{O}$, and $G_0,\,\nabla G_0$ are continuous up to $S$ in $K$.
Let $H_0 := G_0 - \frac{\gamma}{4}|x|^2$. On $S$, define $H_0 = g - \frac{\gamma}{2},\, \nabla H_0 = {\bf v} - \frac{\gamma}{2}x$. Then by (\ref{tangsep2}) we have
$$H_0(\tilde{p}) -H_0(p) - \nabla H_0(p) \cdot (\tilde{p} - p) \geq \frac{\gamma}{4}|\tilde{p} - p|^2$$
for all $p,\, \tilde{p} \in K$. We may thus apply Theorem $1.10$ from \cite{AM} to obtain a global $C^1$, convex function $H_1$ on $\mathbb{R}^2$ such that $H_1 = H_0,\, \nabla H_1 = \nabla H_0$ on $K$.
To finish take $G_1 = H_1 + \frac{\gamma}{4}|x|^2$.
\end{proof}

\subsubsection{Smoothing}

\begin{lem}\label{Smoothing}
There exists a convex function $G \in C^1(\mathbb{R}^2) \cap C^{\infty}(\mathbb{R}^2 \backslash S)$ such that $G = G_1$ in a neighborhood of each point on $\Sigma_0$, and $D^2G \geq \frac{\gamma}{4}I$
on $\mathbb{R}^2$. In particular, $D_G = S$, and $(G,\,\nabla G) = (g,\,{\bf v})$ on $\Sigma_v$.
\end{lem}
\begin{proof}
We begin with a simple observation. Let $F$ be a $C^1$ convex function on $\mathbb{R}^n$ with $D^2F \geq \mu I$.
Let $\rho_{\epsilon}$ be a standard mollifier, and let $F_{\epsilon} := \rho_{\epsilon} \ast F$. Then if $\varphi \in C^{\infty}_0(B_R)$ with $0 \leq \varphi \leq 1$ we have
\begin{align*}
\|(\varphi F_{\epsilon} + (1-\varphi) F) - F\|_{C^1(\mathbb{R}^n)} &= \|\varphi(F_{\epsilon} - F)\|_{C^1(\mathbb{R}^n)} \\
&\leq C(\varphi)\|F_{\epsilon} - F\|_{C^1(B_R)}
\end{align*}
and
\begin{align*}
D^2(\varphi F_{\epsilon} + (1-\varphi) F) &= \varphi D^2 F_{\epsilon} + (1-\varphi) D^2 F + D^2\varphi\,(F_{\epsilon} - F) \\
&+ \nabla \varphi \otimes (\nabla F_{\epsilon} - \nabla F) + (\nabla F_{\epsilon} - \nabla F) \otimes \nabla \varphi \\
&\geq \left(\mu- C(\varphi)\|F_{\epsilon} - F\|_{C^1(B_R)}\right)I.
\end{align*}
Since $F$ is $C^1$, by taking $\epsilon(\varphi)$ small we have that $\varphi F_{\epsilon} + (1-\varphi) F$ is as close as we like to $F$ in $C^1(\mathbb{R}^n)$,
and we have a lower bound for $D^2(\varphi F_{\epsilon} + (1-\varphi) F)$ that is as close as we like to $\mu I$. We will apply this observation to a sequence of mollifications of $G_1$.

Let $\{B_{r_i}(x_i)\}_{i = 1}^{\infty}$ be a Whitney covering of $\mathbb{R}^2 \backslash \Sigma_v$. That is, for 
$$r(x) := \min\{1,\,d_{\Sigma_v}(x)\}/20,$$ 
we have
$r_i = r(x_i)$, the balls $B_{r_i}(x_i)$ are disjoint, $\cup_i B_{5r_i}(x_i) = \mathbb{R}^2 \backslash \Sigma_v$, and for each $x \in \mathbb{R}^2 \backslash \Sigma_v$ the ball $B_{10r(x)}(x)$ intersects at most $(129)^2$ of 
the $\{B_{10r_i}(x_i)\}$. For the existence of such a covering see for example \cite{EG}. 

Now take $\varphi_i \in C^{\infty}_0(B_{10r_i}(x_i))$ such that $\varphi_i = 1$ in $B_{5r_i}(x_i),\, 0 \leq \varphi \leq 1$.
For $j \geq 1$ we define $G_{j+1}$ inductively as follows: If $B_{10r_j}(x_j) \subset \mathcal{U}$ then let $G_{j+1} = G_j$. If not, then let
$$G_{j+1} = \varphi_j (\rho_{\epsilon_j} \ast G_j) + (1-\varphi_j) G_j,$$
with $\epsilon_j$ chosen so small that 
\begin{equation}\label{C1Pert}
\|G_{j+1} - G_j\|_{C^1(\mathbb{R}^2)} \leq 2^{-j}
\end{equation}
and
\begin{equation}\label{C2Pert}
D^2G_{j+1} \geq \left(\frac{\gamma}{2} - \frac{\gamma}{4}\sum_{i = 1}^j 2^{-i}\right)I.
\end{equation}

By the covering properties, for any $x \in \mathbb{R}^2 \backslash \Sigma_v$, we have that $G_j$ are smooth and remain constant in $B_{10r(x)}(x)$ for all $j$ sufficiently large. In
addition, for any point on $\Sigma_0$, every $B_{10r_j}(x_j)$ that intersects a small neighborhood of this point is contained in $\mathcal{U}$, so $G_j = G_1$ in a small neighborhood of every point in $\Sigma_0$. 
We conclude using (\ref{C1Pert}) and (\ref{C2Pert}) that $G := \lim_{j \rightarrow \infty} G_j$ is $C^1$, smooth on $\mathbb{R}^2 \backslash S$, agrees with $G_1$ in a neighborhood of every point on $\Sigma_0$, and satisfies
$$D^2G \geq \left(\frac{\gamma}{2} - \frac{\gamma}{4} \sum_{i = 1}^{\infty} 2^{-i}\right)I = \frac{\gamma}{4}I.$$
\end{proof}



\section*{acknowledgements}
I am grateful to Paolo Marcellini for his interest in this work, and for bringing to my attention some important references on the Lipschitz continuity of minimizers with nonstandard
growth conditions.




\end{document}